\documentclass[12pt]{amsart}
\usepackage[textwidth=16cm, textheight=22cm]{geometry}
\usepackage{graphicx}
\usepackage{latexsym}
\newtheorem{lemma}{Lemma}
\newtheorem{definition}{Definition}
\newtheorem{theorem}{Theorem}
\newtheorem{proposition}{Proposition}

\newtheorem{corollary}{Corollary}

\newtheorem{remark}{Remark}

\def\mmset{{\mathcal B}}

\def\R{\mathbb{R}}

\def\bez{\backslash}
\def\podx{\underline{x}}
\def\nadx{\overline{x}}

\begin{document}

\author{Viorel Nitica}
\address{Department of Mathematics, West Chester University, PA
19383, USA, and Institute of Mathematics, P.O. Box 1-764, Bucharest, Romania}
\email{vnitica@wcupa.edu}

\author{Serge\u{\i} Sergeev}
\address{University of Birmingham,
School of Mathematics, Watson Building, Edgbaston B15 2TT, UK}
\email{sergiej@gmail.com}

\thanks{This research is supported by
NSF grant DMS-0500832 (V. Nitica) as well as  EPSRC grant
RRAH12809 and RFBR grant 08-01-00601 (S. Sergeev)}
\title{An interval version of separation by
semispaces in max-min convexity}

\date{}
\maketitle

\begin{abstract}
In this paper we study separation of a closed box
from a max-min convex set by max-min semispaces.
This can be regarded as an interval extension of the known
separation results. We give a constructive proof of the separation
in the case when the box satisfies a certain condition, and
we show that the separation is never possible when the
condition is not satisfied. We also study the separation
of two max-min convex sets by a box and by a box and a semispace.
\end{abstract}

\section{Introduction\label{sec1}}

Consider the set $\mmset=[0,1]$ endowed with the operations $\oplus
=\max ,\wedge =\min$. This is a well-known
distributive lattice, and like any distributive
lattice it can be considered as a semiring equipped
with addition $\oplus$
and multiplication $\otimes:=\wedge$. Importantly, both operations are idempotent,
$a\oplus a=a$ and $a\otimes a=a\wedge a=a$, and closely related to the
order: $a\oplus b=b\Leftrightarrow a\leq b\Leftrightarrow
a\wedge b=a$.
For standard literature on lattices and semirings
see e.g. \cite{Bir:93} and \cite{Gol:00}.

We consider $\mmset^n$, the cartesian product of $n$ copies of $\mmset$, and
equip this cartesian product with the operations of
taking componentwise $\oplus$: $(x\oplus y)_i:=x_i\oplus y_i$ for $x,y\in\mmset^n$ and $i=1,\ldots, n$,
and scalar $\wedge$-multiplication:
$(a\wedge x)_i:=a\wedge x_i$ for $a\in\mmset$, $x\in\mmset^n$ and $i=1,\ldots,n$. Thus
$\mmset^n$ is considered as a semimodule over $\mmset$ \cite{Gol:00}. Alternatively, one may
think in terms of vector lattices \cite{Bir:93}.

A subset $C$ of $\mmset^{n}$ is said to be {\em max-min convex} if the relations
$x,y\in C,\alpha ,\beta \in \mmset,\alpha \oplus \beta
=1$ imply $(\alpha
\wedge x)\oplus (\beta \wedge y)\in C$.

The interest in max-min convexity is motivated by the
study of tropically convex sets, analogously defined over the semiring
$\R_{\max}$, which is the completed set of real numbers $\R\cup\{-\infty\}$ endowed
with operations of idempotent addition $a\oplus b:=\max(a,b)$ and multiplication
$a\otimes b:=a+b$. Constructed in
\cite{Zim-77,Zim-81}, tropical convexity and its lattice-theoretic
generalizations received much attention and rapidly
developed over the last decades
\cite{AGK-09,CGQS-05,DS-04,GK-06,LMS-01,NS-07I,NS-07II}.
Another source of interest comes from the matrix algebra developed over the
max-min semiring, see \cite{Cec-92,Gav:04,Sem-06} and references therein.

In this article we continue the study of max-min convex structures started
in \cite{NS-08I,NS-08II,Nit-09,Nit-Ser}. We are interested in separation
of max-min convex sets by semispaces.

The set
\begin{align}
[x,y]_M &= \{(\alpha \wedge x)\oplus (\beta \wedge y)\in \mmset%
^{n}|\,\alpha ,\beta \in \mmset,\alpha \oplus \beta =1\}=  \notag \\
\ & =\{\max \,(\min (\alpha ,x),\min (\beta ,y))\in \mmset%
^{n}|\,\alpha ,\beta \in \mmset,\max \,(\alpha ,\beta )=1 \},
\label{segm0}
\end{align}
is fundamental for max-min convexity, it is
called the \emph{max-min segment} (or briefly, the \emph{segment})
\emph{joining }$x$\emph{\ and }$y.$ As in the
ordinary convexity in
the real linear space, a set is max-min convex if and only if any two points
are contained in it together with the
max-min segment joining them. The max-min segments have been described
in \cite{NS-08I,Ser-03}.

Other types of convex sets are max-min semispaces, hemispaces,
halfspaces and hyperplanes
\cite{NS-08II,Nit-09,Nit-Ser}.

For $z\in \mmset^{n},$\ we call
a subset $S$ of $\mmset^{n}$\ a \emph{max-min semispace} (or,
briefly, a \emph{semispace}) \emph{at} $z,$ if it is a
maximal (with respect
to set-inclusion) max-min convex set avoiding $z$.
Semispaces come from the abstract convexity, see e.g. \cite{Sin:97}.
One of their main application is in separation results:
the family of semispaces is the
smallest intersectional basis for the family of all
convex sets. We recall that in $\mmset^{n}$
there exist at most $n+1$ semispaces at each point, exactly $n+1$
at each finite point, and each convex set avoiding $z$ is contained in at
least one of those semispaces \cite{NS-08II}.

Another object introduced in abstract convexity is the {\em hemispace}:
this is any convex set whose complement is also convex.
When the system of convex sets satisfies
Pasch axiom, which is the case for max-min convexity
\cite{NS-08I}, a theorem of Kakutani tells us that for any
two nonintersecting convex sets $C_1$ and $C_2$ there exists a hemispace
$H$ containing $C_1$ such that the complement of $H$ contains
$C_2$. In general, the proof of Kakutani theorem is non-constructive
and uses Zorn's Lemma. A constructive proof of this theorem
in max-min convexity is, to the authors' knowledge,
an open problem.

Hyperplanes and halfspaces are defined by linear forms (in our case,
max-min linear). In the max-min case, these sets are in general
unrelated to semispaces and they cannot
separate a point from a max-min convex set \cite{Nit-09,Nit-Ser}.
This is in contrast with very optimistic results
in the tropical convexity and its lattice-theoretic
generalizations \cite{CGQS-05,DS-04,GK-06,GS-08,Zim-77},
which behave like the ordinary convexity in linear spaces in this respect.

\begin{figure}
\centering
\includegraphics[width=12cm]{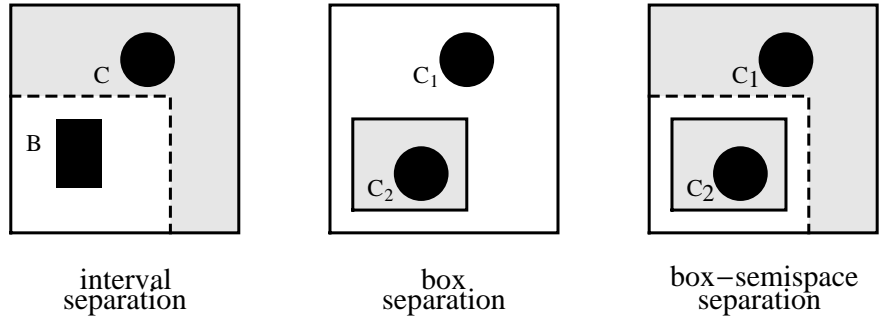}
\caption{Separation types, $n=2$}
\end{figure}

In this paper we study the following interval version of the
semispace separation: given a box $B$, i.e. a Cartesian product of
closed intervals, and a max-min convex set $C$, decide whether it is possible
to construct a semispace which contains $C$ and avoids $B$.
In Section 2 we give our main result, Theorem \ref{interval-sep},
which shows that such separation is indeed
possible when $B$ satisfies
a certain condition. This condition holds true in particular when
$B$ does not contain points with coordinates equal to $1$, or when
$B$ is reduced to a point. When the condition is not satisfied,
we show that the separation by semispaces is never possible.
However, separation can be saved if we also allow hemispaces of a certain kind.
As a corollary of Theorem \ref{interval-sep}, we also recover
the description of semispaces due to Nitica and Singer
\cite{NS-08II}. In Section 3 we study the separation of two convex
sets by a box and by a box and a semispace. We show that this separation is always possible
in $\mmset^2$, and we provide a counterexample in $\mmset^3$.

Figure 1 summarizes the types of separation considered in this paper. The convex sets that need to
be separated are colored in black, and the separating boxes or semispaces are colored in gray.
The sets $C_1,C_2$ and $C$ are convex and $B$ is a box.

The tropical interval linear algebra was introduced in
\cite{LS-01} and \cite[Chapter 6]{intbook:06}.
From this point of view, the present paper may be seen as related to
yet undeveloped area of the interval tropical convexity.

\section{Separation of boxes from max-min convex sets}

For any point $x^0=(x_{1}^{0}, \dots, x_{n}^{0})\in \mmset^n$ we define a family of sets $S_0(x^0),\dots,S_n(x^0)$ in $\mmset^n$. The sets are introduced in \cite[Proposition 4.1]{NS-08II}. Recall that $x^0$ is called {\em finite} if it
has all coordinates different from zeros and ones. Without loss of generality we may assume that:
\begin{equation}
x_{1}^{0}\geq \dots \geq x_{n}^{0}.  \label{decr}
\end{equation}

The set $\{x_{1}^{0},\dots ,x_{n}^{0}\}$
admits a natural subdivision into ordered
subsets such that the elements of each subset are either equal to each other
or are in strictly decreasing order, say
\begin{equation}\label{permut5}
\begin{gathered}
x_{1}^{0}
=\dots =x_{k_{1}}^{0}>\dots >x_{k_{1}+l_{1}+1}^{0}=\dots =x_{k_{1}+l_{1}+k_{2}}^{0}>\dots \\
>x_{k_{1}+l_{1}+
k_{2}+l_{2}+1}^{0}=\dots =x_{k_{1}+l_{1}+k_{2}+l_{2}+k_{3}}^{0}>\dots\\
>x_{k_{1}+l_{1}+\dots +k_{p-1}+l_{p-1}+1}^{0}=\dots =x_{k_{1}+l_{1}+\dots +k_{p-1}+l_{p-1}+k_{p}}^{0}\\
>\dots >x^0_{k_{1}+l_{1}+\dots +k_{p}+l_{p}}(=x^0_n).
\end{gathered}
\end{equation}

Let us introduce the following notations:
\begin{eqnarray}
L_{0} &=&0,K_{1}=k_{1},L_{1}=K_{1}+l_{1}=k_{1}+l_{1},  \label{newnot} \\
K_{j} &=&L_{j-1}+k_{j}=k_{1}+l_{1}+...+k_{j-1}+l_{j-1}+k_{j}\quad
(j=2,...,p),  \label{newnot2} \\
L_{j} &=&K_{j}+l_{j}=k_{1}+l_{1}+...+k_{j}+l_{j}\quad (j=2,...,p);
\label{newnot3}
\end{eqnarray}
we observe that $l_{j}=0$ if and only if $K_{j}=L_{j}.$

We are ready to define the sets. We need to distinguish the cases when the
sequence \eqref{permut5} ends with zeros or begin with ones, since some sets $S_i$ become empty in that case.

\begin{definition}\label{def-sets9}
a) If $x^{0}$ is finite, then:
\begin{equation}
S_{0}(x^0)=\{x\in \mmset^{n}|x_{i}>x_{i}^{0}\text{ for some }1\leq i\leq
n\},  \label{semiunu}
\end{equation}
\begin{equation}\label{semidoi}
\begin{gathered}
S_{K_{j}+q}(x^0)=\{x\in \mmset^{n}|x_{K_{j}+q}<x_{K_{j}+q}^{0},\text{ or }%
x_{i}>x_{i}^{0}\text{ for some }K_{j}+q+1\leq i\leq n\} \\(q=1,...,l_{j};j=1,...,p)\text{ if }l_{j}\neq 0,
\end{gathered}
\end{equation}
\begin{equation}\label{semitrei}
\begin{gathered}
S_{L_{j-1}+q}(x^0)=\{x\in \mmset^{n}|x_{L_{j-1}+q}<x_{L_{j-1}+q}^{0},\text{
or }x_{i}>x_{i}^{0}\text{ for some }K_{j}+1\leq i\leq n\} \\(q=1,...,k_{j};j=1,...,p\text{ if }k_{1}\neq 0,\text{ or }j=2,...,p\text{ if
}k_{1}=0).
\end{gathered}
\end{equation}

b) If there exists an index $i\in \{1,...,n\}$\ such that $x_{i}^{0}=1,$ but no index $j$ such that $x_{j}^{0}=0,$ then the sets are $S_{1},...,S_{n}$ of part a).

c) If there exists an index $j\in \{1,...,n\}$\ such that $x_{j}^{0}=0,$ but no index $i$ such
that $x_{i}^{0}=1,$ then the sets are $S_{0},S_{1},...,S_{\beta -1}$ of part \emph{a)}, where
\begin{equation}
\beta :=\min \{1\leq j\leq n|\;x_{j}^{0}=0 \}.  \label{beta}
\end{equation}

d) If there exist an index $i\in \{1,...,n\}$\ such that $x_{i}^{0}=1,$
and an index $j$ such that $x_{j}^{0}=0,$ then
the sets are $S_{1},...,S_{\beta -1}$ of part \emph{a)},
\emph{\ }where $\beta $ is given by \emph{(\ref{beta})}.
\end{definition}

\begin{proposition}[\cite{NS-08II}]
\label{p:semisp-conv}
For any $x^0\in \mmset^n$ the sets $S_i(x^0), 1\le i\le n,$ are
max-min convex.
\end{proposition}

In the following $[a,c]$ denotes the ordinary interval on the real line
$\{b\colon a\leq b\leq c\}$, provided $a\leq c$ (and possibly
$a=c$).

We investigate the separation of a box
$B=[\podx_1,\nadx_1]\times\ldots\times[\podx_n,\nadx_n]$
from a max-min convex set $C\subseteq\mmset^n$, by which
we mean that there exists a set $S$ described in
Definition \ref{def-sets9},
which contains $C$ and avoids $B$.

Assume that $\nadx_1\geq\ldots\geq\nadx_n$ and suppose that
$t(B)$ is the greatest integer such that
$\nadx_{t(B)}\geq \podx_i$ for all $1\leq i\leq t(B)$.
We will need the following condition:
\begin{equation}
\label{sep-cond}
\begin{split}
&\text{ If }(\nadx_1=1 )\ \&\ (y_l\geq\podx_l,1\le l\le n)\ \&\
(\nadx_l<y_l\ \text{for some $l\leq t(B)$}),\\
&\text{ then } y\notin C.
\end{split}
\end{equation}
Note that if the box is reduced to a point and if $\nadx_1=1$, then
$\nadx_l=1$ for all $l\leq t(B)$ so that $\nadx_l<y_l$ is
impossible. So \eqref{sep-cond} always
holds true in the case of a point.

The formulation of our main result will also use an {\em
oracle} answering the question, whether or not
a given max-min convex set $C\subseteq\mmset^n$ lies in a semispace
$S$.
As in the conventional convex geometry
or tropical convex geometry, this question
can be answered in $O(mn)$ time if $C$ is a convex hull of
$m$ points. Indeed it suffices to answer whether any of the
inequalities defining
$S$ is satisfied for each of the $m$ points generating $C$.

\begin{theorem}
\label{interval-sep}
Let $B=[\podx_1,\nadx_1]\times\ldots\times[\podx_n,\nadx_n]$,
and let $C\subseteq\mmset^n$ be a max-min convex set avoiding
$B$. Suppose that $B$ and $C$ satisfy \eqref{sep-cond}.
Then there is a set $S$ described by
Definition \ref{def-sets9}, which contains $C$ and avoids $B$.
This set is constructed in no more than $n+1$ calls
to the oracle.
\end{theorem}

\begin{proof}
If $\nadx_i<1$ for all $i$, then
we try to separate $B$ from $C$ by $S_0(\nadx_1,\ldots,\nadx_n)$ given by \eqref{semiunu}.
Suppose we fail. Then there
exists $y\in C$ such that $y_i\le\nadx_i$
for all $i$.

Otherwise, $1=\nadx_1\geq\ldots\geq\nadx_n$.
Let $\podx_l=\max\limits_{k\leq t(B)} \podx_k$, and define
$u\in\mmset^n$ by
\begin{equation}
\label{u:def}
u_i=
\begin{cases}
\podx_l, &\text{if $i\leq t(B)$},\\
\nadx_i, &\text{if $i>t(B)$}.
\end{cases}
\end{equation}

It follows from the definition of $t(B)$ that $\podx_l=\max\limits_{1\le i\le n}u_i$. 

We try to separate $B$ from $C$ by $S_l(u)$, which is given by
\eqref{semidoi} or \eqref{semitrei}. If we fail then there
exists $y\in C$ such that $y_i\leq\nadx_i$ for all $i>t(B)$ (and
trivially $y_i\leq\nadx_i$ for $\nadx_i=1$).

Thus we either separate $C$ from $B$, or
there is a point $y\in C$ such that $y_i\leq\nadx_i$ for all
$i>t(B)$. Condition \eqref{sep-cond} and $B\cap C=\emptyset$
assure that there is at least one $i$ such that $y_i<\podx_{i}$.
Indeed, otherwise if $\podx_{i}\le y_i$ for all $i$ and $y_i\leq\nadx_i$ for all
$i\le t(B)$ then $y\in B$; and if $\podx_{i}\le y_i$ for all $i$ and $\nadx_i<y_i$
for some $i\le t(B)$ then $y\not \in C$ by condition \eqref{sep-cond}.

Now assume without loss of
generality that $\podx_1\geq\ldots\geq\podx_n$
(the order of $\nadx_i$ is now arbitrary).

The set $\{1,\ldots,n\}$ is naturally partitioned
by the following procedure. See Figure 2 for an illustration. The segments $[\podx_i,\nadx_i]$ are drawn vertically
and counted from left to right.

\begin{figure}
\centering
\includegraphics[width=10cm]{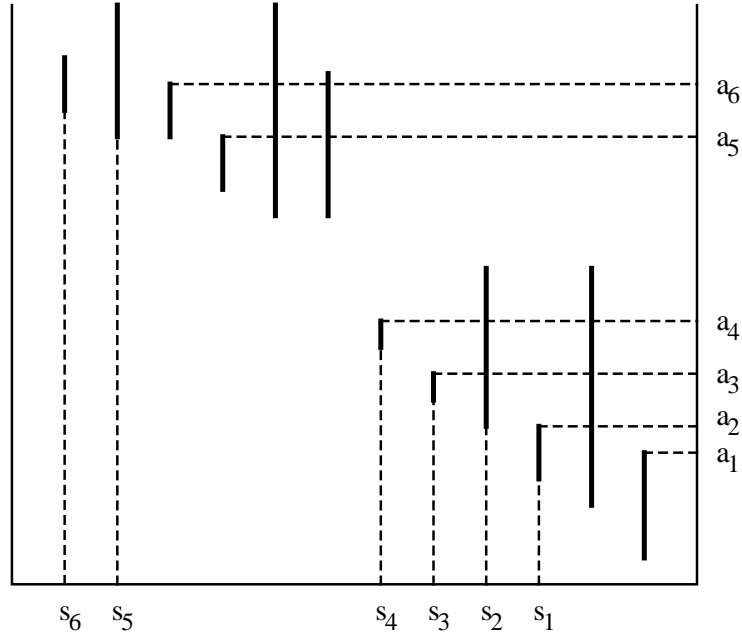}
\caption{The choice of $s_n$}
\end{figure}

Let $s_1$ be the smallest number
such that $\podx_{s_1}\leq \nadx_i$ for all $i=s_1,\ldots,n$.

If $s_1\neq 1$ then there exists
$t_1\in\{s_1,\ldots,n\}$ such that $\podx_{s_1-1}>\nadx_{t_1}$.
In this case let $T_1$ be the set of such $t_1$. Otherwise
if $s_1=1$ we take $T_1=\{1,\ldots,n\}$. In Figure 2 one has $T_1=\{12\}$.

We define
\begin{equation}
\label{a1-def}
a_1=\min\{\nadx_i\colon i\in T_1\}.
\end{equation}
We have
\begin{equation}
\label{a1-props}
\podx_i\leq a_1\leq \nadx_i\quad\forall i=s_1,\ldots, n.
\end{equation}

Thus $a_1$ is a common level in all intervals
$[\podx_{s_1},\nadx_{s_1}],\ldots,[\podx_n,\nadx_n]$, but not
$[\podx_{s_1-1},\nadx_{s_1-1}]$.

If $s_1=1$ then we stop. Otherwise
we proceed by induction. Let $s_k$ be the smallest number such that
$\podx_{s_k}\leq\nadx_i$ for all
$i\in\{s_k,\ldots,n\}\bez T_1\cup\ldots\cup T_{k-1}$.
Note that $s_k<s_{k-1}$. If $s_k\neq 1$ then there exists
$t_k\in\{s_k,\ldots,n\}$ such that
$\podx_{s_k-1}>\nadx_{t_k}$. In this case let $T_k$ be the set of such $t_k$.
Otherwise if $s_k=1$, then define $T_k:=\{1,\ldots,n\}\bez T_1\cup\ldots\cup T_{k-1}$.

We take
\begin{equation}
\label{ak-def}
a_k=\min\{\nadx_i\colon i\in T_k\}.
\end{equation}

We have
\begin{equation}
\label{ak-props}
\podx_i\leq a_k\leq \nadx_i\quad\forall i=\{s_k,\ldots n\}\bez T_1\cup\ldots\cup T_{k-1}.
\end{equation}
Thus $a_k$ is a common level in all intervals
$[\podx_{s_k},\nadx_{s_k}],\ldots,[\podx_n,\nadx_n]$ excluding
the intervals with indices in $T_1\cup\ldots\cup T_{k-1}$ which are below
that level. The interval
$[\podx_{s_k-1},\nadx_{s_k-1}]$ is above $a_k$
(and possibly several other such levels going into $T_k$).

In Figure 2, the sets $T_i$ are $T_1=\{12\},T_2=\{10\},
T_3=\{8\},T_4=\{7,9,11\},T_5=\{4\},T_6=\{1,2,3,5,6\}$.

Next we recall our point $y\in C$. It has
$y_i<\podx_i$ for some $i$. Denote $K=\{i\colon y_i>\nadx_i\}$.
Pick the greatest $i$ such that $y_i<\podx_i$ (note that
for such $i$ we necessarily have $\podx_i>0$),
and let $s_k\leq i< s_{k-1}$, which implies
$\podx_j\leq y_j\leq \nadx_j$ for all $j\in\{s_{k-1},\ldots,n\}\bez K$.
We try to separate
$B$ from $C$ by the sets
\begin{equation}
\label{sipi1r}
S_i(u^i)=
\{x\in\mmset^n\colon x_i<\podx_i\ \text{or}\ x_j>\nadx_j\ \text{for some}\ j\in T_1\cup\ldots\cup T_{k-1}\},
\end{equation}
where $u^i$ can be defined by
\begin{equation}
\label{zi-def}
u_l^i=
\begin{cases}
\podx_l, & l<i,\\
\podx_i, & l\geq i\ \text{and}\ l\notin T_1\cup\ldots\cup T_{k-1},\\
\nadx_l, & l\in T_1\cup\ldots\cup T_{k-1},
\end{cases}
\end{equation}
for all $i$ with $y_i<\podx_i$ and $s_k\leq i<s_{k-1}$. Indeed,
\eqref{sipi1r} is of the form \eqref{semidoi} or \eqref{semitrei},
where $u^i$ is substituted for $x^0$.

Suppose the separation always fails. Then it gives us points $x^i\in C$
such that
\begin{equation}
\label{yi-props}
x_i^i\geq \podx_i\ \text{and}\ x_j^i\leq\nadx_j\
\forall j\in T_1\cup\ldots\cup T_{k-1}.
\end{equation}
Then \eqref{ak-props} implies that
\begin{equation}
\label{akyi-props}
\podx_i\leq a_k\wedge x_i^i\leq\nadx_i\ \text{and}\
a_k\wedge x_j^i\leq \nadx_j\ \forall j=s_k,\ldots,n,
\end{equation}
since $a_k\in[\podx_j,\nadx_j]$ for
$j\in\{s_k,\ldots,n\}\bez T_1\cup\ldots\cup T_{k-1}$ by
\eqref{ak-props}, and we use
\eqref{yi-props} for $j\in T_1\cup\ldots\cup T_{k-1}$.
The point
\begin{equation}
\label{zp-def}
z=\bigoplus_i a_k\wedge x^i\oplus y\in C
\end{equation}
will be in some sense better than $y$. Indeed,
\eqref{akyi-props}
implies that $\podx_i\leq z_i\leq \nadx_i$ for all
$i\in\{s_k,\ldots,n\}\bez K$, versus
$\podx_i\leq y_i\leq \nadx_i$
for all $i=\{s_{k-1},\ldots ,n\}\bez K$. As $z\geq y$ we have
$z_i>\nadx_i$ for all $i\in K$.

Proceeding with this improvement
we obtain a point $z$ which satisfies
$\podx_i\leq z_i$ for all $i$ and $z_i\leq \nadx_i$ for all $i\in\{1,\ldots, n\}\bez K$.
This contradicts either $B\cap C=\emptyset$, or condition
\eqref{sep-cond}.
This contradiction shows that we should
succeed with separation at some stage. Clearly,
the number of calls to the oracle does not exceed
$n+1$.
\end{proof}

We note that Theorem \ref{interval-sep} also yields
a method which verifies condition \eqref{sep-cond} in no
more than $n+1$ calls to the oracle.

The box $B$ can be a point and in this case condition \eqref{sep-cond}
always holds true. Therefore, some known results
on max-min semispaces
\cite{NS-08II} can be deduced from Theorem \ref{interval-sep}.
The following statement is an immediate corollary
of Theorem \ref{interval-sep} and Proposition
\ref{p:semisp-conv}.

\begin{corollary}[\cite{NS-08II}]
\label{ns-08ii}
Let $x\in\mmset^n$ and $C\subseteq\mmset^n$ be a max-min
convex set avoiding $x$. Then $C$ is contained in one $S_i(x), 1\le i\le n,$
as in Definition \ref{def-sets9}. Consequently these
sets are the family of semispaces at $x$.
\end{corollary}
\begin{proof}
The proof of Theorem \ref{interval-sep} applied to $B=\{x\}$ shows
that any max-min convex set avoiding $x$ is contained
in one of the sets $S_i(x)$. Proposition \ref{p:semisp-conv} implies that
these sets are max-min convex and do not contain $x$. Obviously,
they are not included in each other.
If $S_i(x)$ is not maximal, let $S$ be a
max-min convex set strictly containing $S_i(x)$.
Then Theorem \ref{interval-sep} implies
that there exists other $S_j(x), i\not =j,$
such that $S\subset S_j(x)$. But this implies $S_i(x)\subset S_j(x)$, a contradiction.
Hence $S_i(x)$ are all maximal and $\{S_i(x)\}_i$ is the
family of semispaces at $x$.
\end{proof}

Thus we recover a result of \cite{NS-08II}
that Definition \ref{def-sets9}
actually yields all semispaces at a given point.

We now show that separation by semispaces is
impossible when $B$
and $C$ do not satisfy
\eqref{sep-cond}.

\begin{theorem}
\label{BC:nonsep}
Suppose that
$B=[\podx_1,\nadx_1]\times\ldots\times[\podx_n,\nadx_n]$
and max-min convex set $C\subseteq\mmset^n$ are such that
$B\cap C=\emptyset$ but the condition \eqref{sep-cond} does not
hold. Then there is no semispace that contains $C$ and
avoids $B$.
\end{theorem}
\begin{proof}
We assume that $\nadx_1\geq\ldots\geq\nadx_n$.
Since \eqref{sep-cond} does not hold, we have
$\nadx_1=1$. Also there exists $z\in C$, such that
for some indices $k\le t(B)$ we have
\begin{equation}
\label{maxixk}
\max_i\{\podx_i\colon i\leq k\}\leq \nadx_k< z_k,
\end{equation}
but
\begin{equation}\label{someeq33223345}
\podx_i\leq z_i\leq \nadx_i
\end{equation}
for all the others.
Any semispace of the type
$S_0$ given by \eqref{semiunu} intersects with $B$ since
$\nadx_1=1$. If we assume by contradiction that
a separating semispace exists, then it must be of the type
$S_i(x^0)$ given by \eqref{semidoi} or \eqref{semitrei}.
Further, we claim that this semispace must contain the set
$\{x\colon x_k> x_k^0\}$ for some $k$
such that $z_k>x_k^0\geq\nadx_k$,
otherwise it either does not
contain $z$ or it intersects with $B$.
Indeed $z\in S_i(x^0)$ implies
\begin{equation}\label{8833someeq}
z_i<x_i^0,\text{ or }z_j>x_j^0\text{ for }j\text{ such that }x_i^0>x_j^0.
\end{equation}

If $z_j>x_j^0$ is true for some $j$ such that $x_j^0<\nadx_j$, then
$y\in S_i(x^0)$ for each $y\in B$ with $y_j=\nadx_j$, hence
$B\cap S_i(x^0)\neq\emptyset$, a contradiction. If $z_i<x_i^0$ is true,
then $x_i^0>\podx_i$ implying $y\in S_i(x^0)$ for each $y\in B$ with $y_i=\podx_i$,
hence again $B\cap S_i(x^0)\neq\emptyset$, a contradiction.
Thus $z_k>x_k^0$
must hold true for at least one $k$, and necessarily with
$x_k^0\geq\nadx_k$.

This also implies that $i\neq k$, for the type of the
semispace above. Then we must have $x^0_i>x_k^0\geq\nadx_k$ and
$\{x\mid x_i<x_i^0\}\subseteq S_i(x^0)$.
If $i<k$ then $x_i^0>\nadx_k\geq\podx_i$ due to \eqref{maxixk},
and if $i>k$ then also $x_i^0>\nadx_k\geq\nadx_i\geq\podx_i$ by
the ordering of
$\nadx_i$. Hence the set $\{x\colon x_i<x_i^0\}$, which
is contained in $S_i(x^0)$, intersects with $B$
and the separation is impossible.
\end{proof}

\begin{remark} A simple example of interval non-separation is shown in Figure 2.
The box is $B=[0,1]\times [a,b]$ where $0\le a\leq b<1$ and
the convex set is $C=\{z\}$ where $z=(z_1,z_2)$ with $z_2>b$.
Note that $B$ and $C$ do not satisfy condition \eqref{sep-cond}.
\end{remark}

\begin{figure}
\centering
\includegraphics[width=5.5cm]{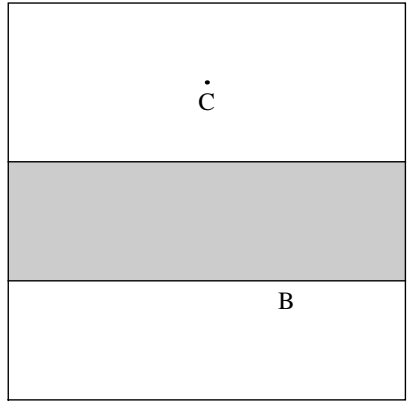}
\caption{Forbidden 2-dim interval separation of convex sets}
\end{figure}

\begin{remark} Theorem \ref{interval-sep} can be easily
modified to allow any case, if in addition to Definition
\ref{def-sets9} we also allow the sets
\begin{equation}
\label{S0m-def}
S_0^M(x^0)=\{x\colon x_i>x^0_i\ \text{for some $i\in M$}\}.
\end{equation}
By Corollary \ref{ns-08ii} these sets cannot be semispaces.
They are {\em hemispaces} in the sense that both the set and
its complement are max-min convex. The condition
$C\subseteq S_0^M(x^0)$ can be verified by the same type of
oracle as in Theorem \ref{interval-sep}.
\end{remark}

\section{Separation of two max-min convex sets}

In this section we investigate the separation of two disjoint closed max-min convex
sets by a box and by a box and a semispace.

We recall the structure of 2-dimensional max-min segments as presented in  \cite{NS-08I}.
Pictures of all types of max-min segments are shown in Figure 4, taken from  \cite{NS-08I}.

\begin{figure}
\centering
\includegraphics[width=8cm]{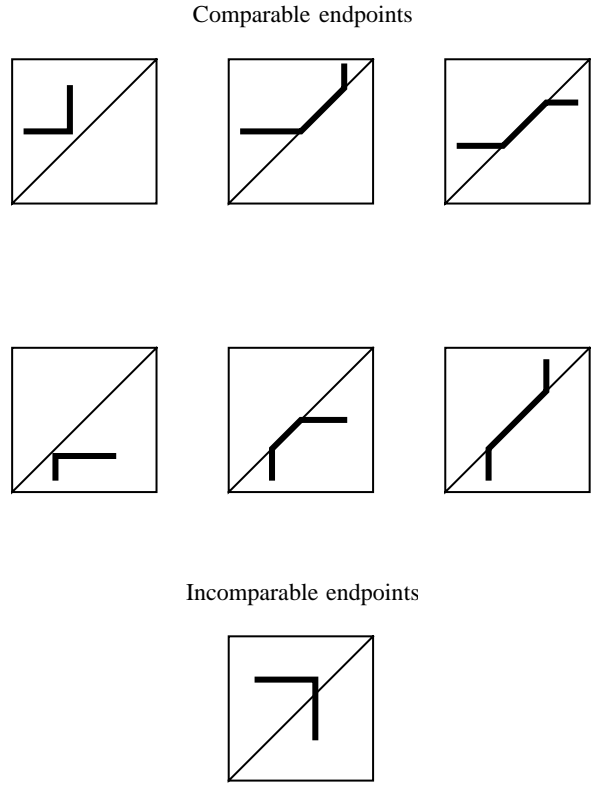}
\caption{2-dim max-min segments}
\end{figure}

\begin{theorem}\label{theor-prelim} Let
$C_1,C_2\in\mmset^2$, $C_1\cap C_2=\emptyset$, be two closed
max-min convex sets. Then there exist a permutation $i:\{1,2\}\to \{1,2\}$ and a box $B\subset \mmset^2$ such that $C_{i(1)}\subset B$ and $B\cap C_{i(2)}=\emptyset$.
\end{theorem}

\begin{proof} Let
\begin{equation}
\begin{gathered}
x_c:=\max\{x\vert (x,y)\in C_1\text{ for some }y\},\\
y_c:=\max\{y\vert (x,y)\in C_1\text{ for some }x\}.
\end{gathered}
\end{equation}
As $C_1$ is compact, there exist $(x_c,y), (x,y_c)\in C_1$, and the convexity of $C_1$ implies that
\begin{equation}
c:=(x_c,y_c)=(x_c,y)\oplus (x,y_c)\in C_1.
\end{equation}

Let
\begin{equation}
\begin{gathered}
x_a:=\min\{x\vert (x,y)\in C_1\text{ for some }y\},\\
y_b:=\min\{y\vert (x,y)\in C_1\text{ for some }x\}.
\end{gathered}
\end{equation}

Consider the points in $C_1$, guaranteed again by compactness:
\begin{equation}
\begin{gathered}
a:=(x_a,y_a),\\
b:=(x_b,y_b).
\end{gathered}
\end{equation}
The values $y_a$ and $x_b$ are chosen arbitrarily.

The smallest box in $\mmset^2$ containing the convex set
$C_1$ is $B_0:=[x_a,x_c]\times [y_b,y_c]$. The point $c$ is the upper right corner of $B_0$.

We need the following Lemma, which can be proved by drawing all possible
special cases and using the structure of max-min segments shown on Figure 4.
This proof is routine and will be omitted.

\begin{lemma}
\label{l:4regions}
The box $B_0$ can be partitioned as $B_0=T_0\cup T_1\cup T_2\cup T_3$,
where
\begin{equation*}
\begin{split}
T_0&=\{\alpha\wedge a\oplus\beta\wedge b\oplus\gamma\wedge c\colon \alpha\oplus\beta\oplus\gamma=1\},\\
T_1&=B_0\cap\{(x,y)\colon x<x_b,\;y<y_a\},\\
T_2&=B_0\cap\{(x,y)\colon y>y_a,\;x<x_c,\;y>x\},\\
T_3&=B_0\cap\{(x,y)\colon x>x_b,\;y<y_c,\;y<x\}.
\end{split}
\end{equation*}
All regions $T_0,T_1,T_2,T_3$ are max-min convex (or possibly empty).
\end{lemma}

The regions $T_0,T_1,T_2,T_3$ are shown in Figure 5.

\begin{figure}
\centering
\includegraphics[width=7cm]{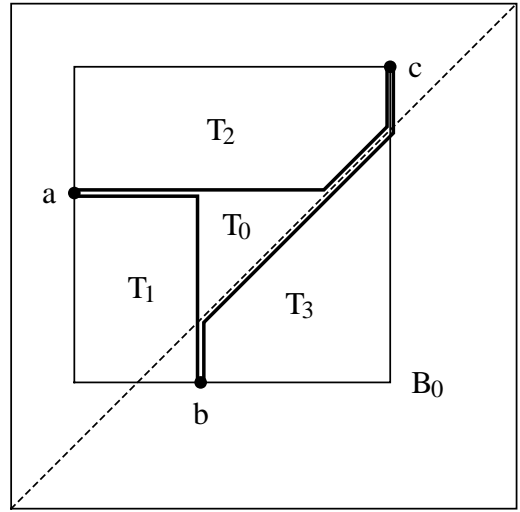}
\caption{2-dim box separation}
\end{figure}

Evidently $T_0\subseteq C_1$ (note that $T_0$ is the
max-min convex hull of $a,b,c$). In particular,
the max-min segments $[a,b]_M,[a,c]_M,[b,c]_M$ are included
in $C_1$ and any point from $C_2$ stays away from them. The other regions
may contain points from both $C_1$ and $C_2$.

We show that if the convex set $C_2$ intersects with
one of the regions $T_1$ and $T_2$, then there is
a box $B_1\subset \mmset^2$
such that $C_2\subset B_1$ and $B_1\cap C_1=\emptyset$.
Due to the symmetry about the main
diagonal, there is no need to consider the case where $C_2$ intersects with $T_3$.

{\em Case 1.} Assume $C_2$ intersects with the region $T_1$.

The intersection $C_2'$ of $C_2$ with
region $T_1$ is max-min convex (as the intersection of
two max-min convex sets). Thus there exists a point
$(x_M,y_M)\in C'_2$, away from the boundary of $C_1$, and
consequently away from the segment $[a,b]_M$, that has the maximum $x$-coordinate and maximum $y$-coordinate for $C_2'$.

We show that $C_2$ is included in the box $B_1=[0,x_M]\times [0,y_M]$.

Assume by contradiction that
there exists $(x',y')\in C_2$ such that $x'>x_M$ or $y'>y_M$, then
$(x'',y''):=(x_M,y_M)\oplus(x',y')$ has either $x''=x_M$ and $y''=y'> y_M$,
or $x''=x'> x_M$ and $y''=y_M$, or $(x'',y'')=(x',y')$.
If $x''<x_b$ and $y''<y_a$ then $(x'',y'')\in C_2'$, which contradicts
the maximality of $x_M$ and $y_M$. Otherwise, the segment
$[(x_M,y_M),(x'',y'')]_M$ intersects with $[a,b]_M$ and hence
$C_1\cap C_2\neq\emptyset$, a contradiction.

We show that the box $B_1$ does not intersect with $C_1$.
Assume that there exists $(x,y)\in B_1\cap C_1$. Then $x\leq x_M$ and $y\leq y_M$.
There exist $(x',y_M)\in[a,b]_M$ and $(x_M,y')\in[a,b]_M$
such that $x'>x_M$ and $y'>y_M$. Using these points, we obtain that
\begin{equation}
\begin{split}
(x_M,y_M)&=(x,y)\oplus x_M\wedge (x',y_M),\ \text{ if } x_M\geq y_M\\
(x_M,y_M)&=(x,y)\oplus y_M\wedge (x_M,y'),\ \text{ if } x_M\leq y_M.
\end{split}
\end{equation}
In both cases $(x_M,y_M)\in C_1$ and hence $C_1\cap C_2\neq\emptyset$,
a contradiction.

{\em Case 2.} Assume now that $C_2$ intersects with $T_2$,
and let $C_2':=C_2\cap T_2$. Let $x_M$ be the largest $x$ coordinate
of a point in $C_2'$ and $y_M$ the lowest $y$-coordinate of a point in $C_2'$. Let $(x_0,y_M), (x_M,y_0)\in C'_2$.
From the definition of $T_2$ we have $x_0\le y_M$ and $x_M\le y_0$. Let $[t_1,t_2]:=[x_a,x_c]\cap[y_a,y_c]$
(where all segments are ordinary
on the real line).

If $y_M\leq x_M$, then due to convexity
\begin{equation}
\begin{gathered}
(x_0,x_0)=(x_0,y_M)\oplus x_0\wedge (x_M,y_0)\in C_2',\\
(y_M,y_M)=(x_0,y_M)\oplus y_M\wedge (x_M,y_0)\in C_2',
\end{gathered}
\end{equation}
and hence the whole diagonal (and max-min) segment $[(x_0,x_0),(y_M,y_M)]_M$
is included in $C_2'$.
It can be observed that any point in the closure of $T_2$ that
belongs to the main diagonal lies in $[a,c]_M$ which is in $C_1$.
Thus $C_1\cap C_2\neq\emptyset$, a contradiction,
hence we must have $y_M>x_M$.

When $y_M>x_M$, due to convexity we have
\begin{equation}
(x_M,y_M)=(x_0,y_M)\oplus y_M\wedge (x_M,y_0)\in C'_2.
\end{equation}
In this case we claim that $C_2$ is contained in the box
$B_1:=[0,x_M]\times [y_M,1]$, which avoids $C_1$.

Assume by contradiction that there exists $(x',y')\in C_2$ which does
not lie in $B_1$. This implies that $x'>x_M$ or $y'<y_M$.
We also have
$y_M>x'$ and $y'>x_M$, otherwise the segment
$[(x',y'),(x_M,y_M)]_M$ has points on the main diagonal, in which
case it intersects with $[a,c]_M$. Consider the combinations
\begin{equation}
\begin{split}
&(x',y_M)=y_M\wedge(x',y')\oplus (x_M,y_M),\ \text{ if } x'>x_M,\\
&(x_M,y')=(x',y')\oplus y'\wedge (x_M,y_M),\ \text{ if } y'<y_M \text{ and }x'\leq x_M.
\end{split}
\end{equation}
Thus we obtain either $(x',y_M)\in C_2$ with $x'>x_M$, or
$(x_M,y')\in C_2$ with $y'<y_M$ and $x'\leq x_M$, leading to
a contradiction with the maximality of $x_M$ or the minimality of
$y_M$.

To prove that $B_1$ avoids $C_1$, assume by contradiction
that there exists $(x,y)\in C_1$ where $x\leq x_M$ and $y\geq y_M$.
We observe that there is a point
$(x_M,y')\in [a,c]_M$, where $y'\leq y_M$. Using this point we obtain
\begin{equation}
(x_M,y_M)=y_M\wedge(x,y)\oplus (x_M,y'),
\end{equation}
which implies $(x_M,y_M)\in C_1$, hence $C_1\cap C_2\neq\emptyset$,
a contradiction.
\end{proof}

\begin{theorem}\label{theorem200} Let
$C_1,C_2\in\mmset^2$, $C_1\cap C_2=\emptyset$, be two closed
max-min convex sets that are away from the boundary of $\mmset^2$.
Then there exist a permutation $i:\{1,2\}\to \{1,2\}$,
a box $B\subset \mmset^2$ and a
semispace $S\subset \mmset^2$ such
that $C_{i(1)}\subset B,C_{i(2)}\subset S$
and $B\cap S=\emptyset$.
\end{theorem}

\begin{proof} The statement follows from Theorem \ref{interval-sep}
and Theorem \ref{theor-prelim}. Indeed, Theorem \ref{theor-prelim}
implies that either the minimal containing box of $C_1$ does not
intersect with $C_2$, or the minimal containing box of $C_2$ does not
intersect with $C_1$.
The condition \eqref{sep-cond} is satisfied due to the
fact that the convex sets are away from the boundary of $\mmset^n$
and hence so are the minimal containing boxes. Applying
Theorem \ref{interval-sep} we obtain the statement.
\end{proof}

\begin{figure}
\centering
\includegraphics[width=6cm]{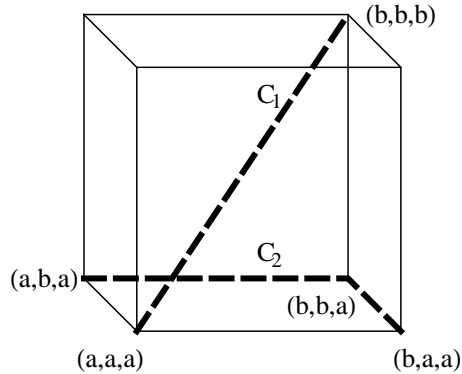}
\caption{Forbidden 3-dim box-semispace separation}
\end{figure}

\begin{remark} We observe that Theorem \ref{theor-prelim} and Theorem \ref{theorem200} are not valid in dimension 3 or higher.

Let $C_1$ be the max-min segment
$[(a,a,a),(b,b,b)]_M$ and $C_2$ be the max-min segment
$[(b,a,a),(a,b,a)]_M$, where $0\leq a<b\leq 1$.
It follows from \cite{NS-08I} that
$C_1$ is part of the main diagonal, and that $C_2$ is the concatenation of two pieces with parametrizations
$\{(t,b,a)\vert a\le t\le b\}$ and $\{(b,t,a)\vert a\le t\le b\}$.
It follows from Figure 6 that the smallest box
containing $C_1$ is $[a,b]^3$
and the smallest box containing $C_2$ is $[a,b]^2\times \{a\}$. Since one box is completely included in the other, box separation or box-semispace separation of $C_1$ and $C_2$ is not possible.
\end{remark}


\begin{thebibliography}{10}



\bibitem{AGK-09}
X.~Allamegeon, S.~Gaubert, and R.~Katz.
\newblock The number of extreme points of tropical polyhedra.
\newblock E-print arXiv:math/0906.3492v1, 2009.

\bibitem{Bir:93}
G.~Birkhoff.
\newblock {\em Lattice theory}.
\newblock American Mathematical Society, Providence, 1993.

\bibitem{Cec-92}
K.~Cechl\'arov\'a.
\newblock Eigenvectors in bottleneck algebra.
\newblock {\em Linear Algebra Appl.}, 175:63--73, 1992.

\bibitem{CGQS-05}
G.~Cohen, S.~Gaubert, J.P. Quadrat, and I.~Singer.
\newblock Max-plus convex sets and functions.
\newblock In G.~Litvinov and V.~Maslov, editors, {\em Idempotent Mathematics
  and Mathematical Physics}, volume 377 of {\em Contemporary Mathematics},
  pages 105--129. AMS, Providence, 2005.
\newblock E-print arXiv:math/0308166.

\bibitem{DS-04}
M.~Develin and B.~Sturmfels.
\newblock Tropical convexity.
\newblock {\em Documenta Math.}, 9:1--27, 2004.
\newblock E-print arXiv:math/0308254.

\bibitem{intbook:06}
M.~Fiedler, J.~Ram{\'\i}k, J.~Nedoma, J.~Rohn, and K.~Zimmermann.
\newblock {\em {Linear optimization problems with inexact data}}.
\newblock Springer Verlag, 2006.

\bibitem{GK-06}
S.~Gaubert and R.~Katz.
\newblock {\em Max-plus convex geometry}, volume 4136 of {\em Lecture Notes in
  Computer Sciences}, pages 192--206.
\newblock Springer, New York, 2006.

\bibitem{GS-08}
S.~Gaubert and S.~Sergeev.
\newblock Cyclic projectors and separation theorems in idempotent convex
  geometry.
\newblock {\em Journal of Math. Sci.}, 155(6):815--829, 2008.
\newblock E-print arXiv:math/0706.3347.

\bibitem{Gav:04}
M.~Gavalec.
\newblock {\em Periodicity in Extremal Algebra}.
\newblock Gaudeamus, Hradec Kr\'alov\'e, 2004.

\bibitem{Gol:00}
J.~Golan.
\newblock {\em Semirings and their applications}.
\newblock Kluwer, Dordrecht, 2000.

\bibitem{LMS-01}
G.L. Litvinov, V.P. Maslov, and G.B. Shpiz.
\newblock Idempotent functional analysis. algebraic approach.
\newblock {\em Math. Notes (Moscow)}, 69(5):758--797, 2001.

\bibitem{LS-01}
G.L. Litvinov and A.N. Sobolevski{\u{\i}}.
\newblock {Idempotent interval analysis and optimization problems}.
\newblock {\em Reliable Computing}, 7(5):353--377, 2001.

\bibitem{Nit-09}
V.~Nitica.
\newblock The structure of max-min hyperplanes.
\newblock Submitted to Linear Algebra Appl., 2009.

\bibitem{Nit-Ser}
V.~Nitica and S.~Sergeev.
\newblock On hyperplanes and semispaces in max-min convex geometry.
\newblock Submitted to Kybernetika, 2009.

\bibitem{NS-07I}
V.~Nitica and I.~Singer.
\newblock Max-plus convex sets and max-plus semispaces. i.
\newblock {\em Optimization}, 56:171--205, 2007.

\bibitem{NS-07II}
V.~Nitica and I.~Singer.
\newblock Max-plus convex sets and max-plus semispaces. ii.
\newblock {\em Optimization}, 56:293--303, 2007.

\bibitem{NS-08I}
V.~Nitica and I.~Singer.
\newblock Contributions to max-min convex geometry. i: Segments.
\newblock {\em Linear Algebra Appl.}, 428(7):1439--1459, 2008.

\bibitem{NS-08II}
V.~Nitica and I.~Singer.
\newblock Contributions to max-min convex geometry. ii: Semispaces and convex
  sets.
\newblock {\em Linear Algebra Appl.}, 428(8-9):2085--2115, 2008.

\bibitem{Sem-06}
B.~Seman\v{c}\'{\i}kov\'{a}.
\newblock Orbits in max-min algebra.
\newblock {\em Linear Algebra Appl.}, 414:38--63, 2006.

\bibitem{Ser-03}
S.~N. Sergeev.
\newblock Algorithmic complexity of a problem of idempotent convex geometry.
\newblock {\em Math. Notes (Moscow)}, 74(6):848--852, 2003.

\bibitem{Sin:97}
I.~Singer.
\newblock {\em {Abstract convex analysis}}.
\newblock Wiley-Interscience, 1997.

\bibitem{Zim-77}
K.~Zimmermann.
\newblock A general separation theorem in extremal algebras.
\newblock {\em Ekonom.-Mat. Obzor (Prague)}, 13:179--201, 1977.

\bibitem{Zim-81}
K.~Zimmermann.
\newblock Convexity in semimodules.
\newblock {\em Ekonom.-Mat. Obzor (Prague)}, 17:199--213, 1981.



\end{thebibliography}
\end{document}